\documentclass[11pt]{amsart}

\usepackage{amsmath,amssymb,amsthm}
\usepackage{fullpage}
\usepackage[all]{xy}
\usepackage{color}
\usepackage{url}

\providecommand{\abs}[1]{\left\lvert#1\right\rvert}

 \DeclareMathOperator{\tor}{Tor}

\DeclareMathOperator{\Span}{Span}

\DeclareMathOperator{\Per}{Per}

\DeclareMathOperator{\Res}{Res}
 
 \DeclareMathOperator{\Pre}{Preper}
\DeclareMathOperator{\hhat}{\hat{h}}

\newcommand{\col}{\,{:}\,}
\newcommand{\tth}{^{\operatorname{th}}}

\theoremstyle{plain}
\newtheorem{thm}{Theorem}
\newtheorem*{thm*}{Theorem}
\newtheorem{lem}{Lemma}
\newtheorem{prop}{Proposition}
\newtheorem*{prop*}{Proposition}
\newtheorem{cor}{Corollary}
\theoremstyle{definition}
\newtheorem{defn}{Definition}
\newtheorem{conj}{Conjecture}
\newtheorem*{conj*}{Conjecture}
\newtheorem*{conj1a}{Conjecture 1a}
\newtheorem{exmp}{Example}
\newtheorem*{exmp*}{Example}
\newtheorem*{rem}{Remark}
\theoremstyle{remark}

\def\Z{\mathbb{Z}}
\def\Q{\mathbb{Q}}
\def\P{\mathbb{P}}
\def\A{\mathbb{A}}

\def\N{\mathbb{N}}

\def\O{\mathcal{O}}

\newcommand{\etalchar}[1]{$^{#1}$}

\providecommand{\bysame}{\leavevmode\hbox to3em{\hrulefill}\thinspace}
\providecommand{\MR}{\relax\ifhmode\unskip\space\fi MR }

\providecommand{\href}[2]{#2}

\begin{document}
    \title{Determination of all rational preperiodic points for morphisms of PN}
    \author{Benjamin Hutz}
    \address{Department of Mathematical Sciences \\
        Florida Institute of Technology \\
        150 W. University Blvd\\
        Melbourne, FL 32901 \\
        USA}
    \email{bhutz@fit.edu}

    \thanks{The author thanks ICERM, where much of this work was completed, and ICERM and the Brown Center for Computing and Visualization for computation time. The current state of the arithmetic dynamics project for Sage, including the implementation of the algorithm discussed in this article, can be found at \url{http://wiki.sagemath.org/dynamics/ArithmeticAndComplex}.}



\begin{abstract}
    For a morphism $f:\P^N \to \P^N$, the points whose forward orbit by $f$ is finite are called \emph{preperiodic points} for $f$. This article presents an algorithm to effectively determine all the rational preperiodic points for $f$ defined over a given number field $K$. This algorithm is implemented in the open-source software Sage for $\Q$. Additionally, the notion of a dynatomic zero-cycle is generalized to preperiodic points. Along with examining their basic properties, these generalized dynatomic cycles are shown to be effective.
\end{abstract}

\maketitle

    Let $f:\P^N \to \P^N$ be a morphism of (algebraic) degree at least $2$ defined over a number field $K$. Let $P \in \P^N(K)$ be a point, then we define the \emph{$n\tth$ iterate of $P$} as
	\begin{equation*}
		f^n(P) = f \circ f^{n-1}(P).
	\end{equation*}
	The collection of iterates
	\begin{equation*}
		\O_f(P) = \{P, f(P), f^2(P), \ldots\}
	\end{equation*}
    is called the \emph{(forward) orbit} of $P$ by $f$. If $\#\O_f(P)$ is finite, we say that $P$ is \emph{preperiodic}. A preperiodic point $P$ is \emph{periodic of period $n$} if $f^n(P)=P$. The smallest such $n$ is called the \emph{minimal period} for $P$. Northcott in the 1950s \cite{Northcott} used height functions to show that for any given $f$ and $K$, the set of preperiodic points of $f$ defined over $K$, $\Pre(f,K)$,  is finite. This is the dynamical analog of the finiteness of the rational torsion subgroup of an abelian variety $A/K$. Morton-Silverman in 1994 \cite{Silverman7} conjectured that $\Pre(f,K)$ is bounded independently of the map $f$, the dynamical analog of Merel's Theorem \cite{Merel} for elliptic curve torsion.
	\begin{conj*}[Morton-Silverman \cite{Silverman7}]
		For any integers $d \geq 2$, $N \geq 1$, and $D\geq 1$ there is a
        constant $C=C(d,N,D)$ with the following property: For any number field $K/\Q$ with $[K \col \Q]\leq D$
        and any morphism $f:\P^N\to\P^N$ of degree~$d$ defined over~$K$,
        \begin{equation*}
          \# \Pre(f,K) \le C.
        \end{equation*}
	\end{conj*}

    Poonen studied the explicit case where $f:\P^1 \to \P^1$ is a degree 2 polynomial defined over $\Q$. He makes the following conjecture.
	\begin{conj*}[Poonen \cite{Poonen}]
        For $n > 3$ there is no quadratic polynomial $f$ defined over $\Q$ with a $\Q$-rational periodic point of minimal period $n$.
	\end{conj*}
    Assuming this conjecture, he shows that there can be at most $9$ $\Q$-rational preperiodic points and classifies all possible graph structures of $\Pre(f,\Q)$. For $n=4$ \cite{Morton4} and $n=5$ \cite{FPS} it is known that there are no $\Q$ rational periodic points of minimal period $n$. Stoll \cite{Stoll3} conditionally proves the case where $n=6$.  Jointly with Patrick Ingram \cite{Hutz2} the author has verified Poonen's conjecture for quadratic polynomials $x^2 + c$ for all $c$ values with numerator and denominator up to $\pm 10^8$.

    A few other families have been studied. Manes \cite{Manes2} studied a certain family of quadratic rational maps on $\P^1(\Q)$ conjecturing maximal period $4$ and at most $11$ rational preperiodic points. In higher dimensions, the author \cite{Hutz3} searched for periodic points defined over $\Q$ on Wehler's class of K3 surfaces defined on $\P^2 \times \P^2$. In \cite{Hutz4} the author constructs families of degree 2 polynomial maps on $\P^N$ which determines a lower bound on the growth factor for the largest minimal period of a $\Q$-rational periodic point as $N$ increases. In all these situations, the methods used were specific to the family studied and, except for \cite{Manes2, Poonen}, focused solely on the periodic points.\footnote{There is a recent preprint by Doyle-Faber-Krumm that presents an algorithm to compute $\Pre(f,K)$ for $f(z) = z^2+c$. Their method is to compute a height bound and then enumerate the points of small height on $K$\cite{DFK}.}
	
    The purpose of this article is to provide an algorithm for a number field $K$ and an implementation for $K=\Q$ in Sage \cite{sage} to compute $\Pre(f,K)$ and to gather numerical evidence related to the conjectures of Morton-Silverman and Poonen. Additionally, the notion of dynatomic cycles \cite{Hutz1} is extended to preperiodic points. Their basic properties are studied and they are shown to be effective.

    The article is organized as follows. Section \ref{sect_problem} discusses the problem and the difficulties the algorithm must solve. Section \ref{sect_algorithm} describes the algorithm in detail. Section \ref{sect_dyn_cycles} describes generalized dynatomic cycles. Section \ref{sect_experimental_results} discusses experimental results from applying the algorithm to various families of maps. Finally, Section \ref{sect_examples} provides a few interesting isolated examples and discusses running time of the implementation and next steps.

\section{Discussion of the problem}\label{sect_problem}
    Given a morphism $f:\P^N(K) \to \P^N(K)$ defined over a number field $K$, the goal is to find all the points in $\P^N(K)$ with finite forward orbit, i.e., the rational preperiodic points. The method is to use information about the cycle structure of $f$ modulo primes to determine information about the cycle structure of $f$ over $K$. This is similar to searching for rational solutions to Diophantine equations by reducing modulo primes.

    When reducing a rational map modulo primes, we must take some care in choosing the primes. Let $R$ be a discrete valuation ring, $K$ its field of fractions, $\pi$ a uniformizer, and $k$ the residue field with characteristic $p$.
	\begin{defn}
	    Denote $\overline{x}$ as the reduction of $x$ mod $\pi$.
	
	    We reduce a polynomial $f$ mod $\pi$, denoted $\overline{f}$, by reducing each of its coefficients.
	\end{defn}

    The notion of good reduction we want is that the local dynamics reflects the global dynamics
    \begin{equation} \label{eq1}
        \overline{f^n(x)} = \overline{f}^n(\overline{x}).
    \end{equation}
    \begin{prop}[{\cite{Silverman6}}] \label{thm_good_red_defn}
        Let $f=(f_0,\ldots,f_N):\P^N_R \to \P^N_R$. The following are equivalent:
        \begin{enumerate}
            \item $\deg(f) = \deg(\overline{f})$
            \item The equations $\overline{f_i}=0$ have no common solutions.
            \item The resultant $\Res(f_0,\ldots,f_n) \not\equiv 0 \pmod{\pi}$. \label{thm_GR_res}
        \end{enumerate}
    \end{prop}
    \begin{defn}
        If $f \col \P^N \to \P^N$ satisfies any condition of Proposition \ref{thm_good_red_defn}, then we say that $f$ has \emph{good reduction} modulo $\pi$. Otherwise, we say $f$ has \emph{bad reduction} modulo $\pi$.

        A map with good reduction satisfies equation (\ref{eq1}).
    \end{defn}

    It is clear that the minimal period of a periodic point in the residue field $k$ (local information) must divide the minimal period over $K$ (global information). Furthermore, there is a precise description of the relationship between the local and global minimal periods. For $f:\P^1 \to \P^1$ this is a collection of results from several authors \cite[Theorem 2.21]{Silverman10}. We state the special case of $f:\P^N \to \P^N$ from the general theorem \cite{Hutz2}.
	\begin{prop}[\cite{Hutz2}] \label{thm_goodreduction}
        Let $f:\P^N \to \P^N$ be a morphism defined over $K$ with good reduction at $\pi$. Let $P \in \P^N(K)$ be a periodic point with minimal period $n$ such that $\overline{P}$ has minimal period $m$ for $\overline{f}$. Then, there is some $f$ stable subspace $V$ of the cotangent space of $\P^N$ such that
        \begin{equation*}
            n=m \quad \text{or} \quad n=mr_Vp^{e}
        \end{equation*}
        for some explicitly bounded integer $e\geq 0$ and where $r_V$ is the order of $d\overline{f^m}_P$ on $V$. The bound for $e$ is given by
         \begin{equation*}
            e \leq
            \begin{cases}
              1 + \log_2(v(p)) & p \neq 2\\
              1 + \log_{\alpha}\left(\frac{\sqrt{5}v(2) + \sqrt{5(v(2))^2 + 4}}{2}\right) & p=2,
            \end{cases}
        \end{equation*}
        where $\alpha = \frac{1+\sqrt{5}}{2}$ and $v(\cdot)$ is the valuation. Moreover, $V$ is the scheme theoretic closure of the finite orbit $\O_{\overline{f}}(\overline{P})$.
    \end{prop}

\section{Algorithm} \label{sect_algorithm}
    A broad outline of the algorithm is as follows.
    \begin{enumerate}
        \item For several primes $p$ with good reduction, find the list of possible global periods:
            \begin{enumerate}
                \item Find all the periodic cycles modulo $p$.
                \item Compute $m, mr_Vp^e$ for each cycle. (Proposition \ref{thm_goodreduction}).
            \end{enumerate}
        \item Intersect the lists of possible periods for the chosen primes.
        \item For each $n$ in the intersection, find all rational solutions to $f^n(P) = P$.

            We have now determined all the rational periodic points.
        \item For each known rational preperiodic point $P$ find all its rational preimages, i.e., rational solutions to $f(Q)=P$.
        \begin{enumerate}
            \item Repeat until there are no new rational preperiodic points.
        \end{enumerate}
    \end{enumerate}
    We now discuss details and give examples.
    \subsection{Rational periodic points}
        As a direct consequence of Proposition \ref{thm_good_red_defn}\ref{thm_GR_res}, each map has a finite number of primes of bad reduction. In practice, the data from a small number of good primes ($3$-$5$) is typically sufficient to have the correct (or nearly correct) set of possible periods.
        Consequently, we assume that the primes used are small. With this assumption, the algorithm is implemented in the memory-intensive approach of building the complete table of forward images. The points of $\P^N(k)$ are hashed so that the table of entries $(P,f(P))$ is a table of integers.
        The iteration data is obtained by evaluating the function at each point only once. This method is used because evaluation of the function is a very expensive operation compared to a table look-up.

        Once all the local cycles have been found, we must compute the complete list of possible periods from Proposition \ref{thm_goodreduction}: $\{m,mr_V,mr_Vp,mr_Vp^2,\ldots, mr_Vp^e\}$.  An easy application of the chain rule allows us to compute the derivative of iterates with only evaluation of the function itself.
        \begin{equation} \label{eq_multiplier}
            (f^m(P))' = \prod_{i=0}^{m-1} f'(f^i(P)).
        \end{equation}
        This is important since for large periods the function $f^m$ will become unmanageable.

        We must also take into account the $V$ in Proposition \ref{thm_goodreduction}. In dimension $1$, there is no issue as $\P^1$ and $V$ are both dimension $1$, and we simply take the multiplicative order of the value of the derivative (i.e., the multiplier).
        However, in higher dimensions the derivative becomes the Jacobian matrix; and we need the multiplicative order of the, potentially proper, subset $V$ of that matrix. While $V$ may be explicitly described as the scheme theoretic closure of a finite list of points, determining the explicit subset of the derivative matrix can be time consuming.
        Since, in practice, the data from a small number of primes greatly reduces the list of possible periods, this issue was by-passed. We compute the multiplicative order of the eigenvalues of the Jacobian matrix and then expand the list $\{m,mr,mrp, \ldots, mrp^e\}$ by allowing $r$ to be the least common multiple of \emph{any} combination of those values.
        While this expands the list of possible periods for each prime, it is still a valid list of possible periods since the correct $r_V$ must be contained in the list of least common multiples.

        \begin{exmp}
            For the map $f(z) = z^2-7/4$ over $\Q$, the only prime of bad reduction is $2$. Reducing modulo $3$, $(1:0)$ is fixed and $(0:1)$ is periodic of minimal period 2. They both have multiplier 0, so the possible global periods are
            \begin{equation*}
            		\Per_3=\{1,2\}.
            \end{equation*}
            Reducing modulo $5$, $(1:0)$ is fixed and $(1:1)$ is periodic of minimal period 2. The point $(1 : 1)$ has $r=4$, so the possible global periods are
            \begin{equation*}
            		\Per_5=\{1,2,8\}.
            \end{equation*}
            Reducing modulo $7$, the possible global periods are
            \begin{equation*}
            		\Per_7=\{1,2,3,6\}.
            \end{equation*}
            The intersection of these sets of possible periods is
            \begin{equation*}
            		\Per_3 \cap \Per_5 \cap \Per_7 = \{1,2\}
            \end{equation*}
            Thus, over $\Q$, there may be fixed points and points of minimal period $2$. This does not guarantee that there are points with these periods, but it does prove that there are not any (rational) points of any other minimal period.

            Solving the two equations
            \begin{align*}
                f(z) = z \qquad f^2(z) = z
            \end{align*}
            we find the fixed point at infinity $(1:0)$ and the two-cycle $[(1:2) \to (-3:2)]$.
        \end{exmp}

        To determine the rational periodic points, we find all rational solutions to the equations $f^n(P) = P$ for all $n$ in the set of possible periods. The implementation computes a $p$-adic approximation of local periodic points to an accuracy predetermined by a height calculation. The smallest rational point approximated by this $p$-adic approximation is determined by the $LLL$ basis reduction algorithm \cite{LLL}. If this point has height larger than the precomputed bound, then it is not a periodic point. Otherwise, we verify that it is a periodic point.  Over number fields, there are other algorithms to compute ``short'' bases \cite{Fieker}.

        To produce a height bound, we use a Nullstellensatz argument. In particular, if $f$ is a morphism then there is some integer $D$ such that
        \begin{equation*}
            x_i^D \in (f_0,\ldots,f_N) \qquad 0 \leq i \leq N.
        \end{equation*}
        A value of $D$ valid for all $f:\P^N \to \P^N$ is explicitly known.
		\begin{lem}{\cite[Corollary p.169]{Lazard}}
			If $f=[f_0,\ldots,f_N]$ is a morphism with $f_i \in K[x_0,\ldots,x_N]$, then
			\begin{equation*}
				x_i^{(N+1)(d-1)+1} \in (f_0,\ldots,f_N).
			\end{equation*}
		\end{lem}
        We then explicitly find the combinations
        \begin{equation*}
            x_j^D = \sum_{i=0}^N f_{i}g_{i,j} \qquad 0 \leq j \leq N
        \end{equation*}
        and use them to compute an explicit bound on the difference between the height of a point and the canonical height of a point. Since the canonical height of a preperiodic point must be $0$, this gives an upper bound on the height of a preperiodic point.
        \begin{defn}
            For a polynomial $f(x_0,\ldots,x_N) = \sum_{\alpha} c_{\alpha}x^{\alpha}$, we define its height as the maximum height of its coefficients
            \begin{equation*}
                h(f) = \max_{\alpha}(h(c_{\alpha})).
            \end{equation*}
        \end{defn}

        \begin{prop} \label{prop_height_bound}
            Let $f:\P^N \to \P^N$ be a degree $d$ morphism and $P \in \Pre(F,K)$. Let
            \begin{equation*}
                D = (N+1)(d-1)+1.
            \end{equation*}
            Then,
            \begin{equation*}
                h(P) \leq \frac{1}{\deg(f) - 1}\max\left(h(f)+\log{\binom{N+d}{d}}, \log\left((N+1)\binom{N+D-d}{D-d}\right) + \max_i h(g_i)\right).
            \end{equation*}
        \end{prop}
        \begin{proof}
            We need the constant $C$ such that
            \begin{equation*}
                \abs{\hhat_f(P) - h(P)} < C,
            \end{equation*}
            where $\hhat_f$ is the canonical height of $P$ with respect to $f$.

            We first produce a bound
            \begin{equation*}
                \abs{h(f(P)) - dh(P)} \leq C_1.
            \end{equation*}
            An upper bound is obtained by simply taking the largest coefficient of $f$ times the number of monomials of degree $d$,
            \begin{equation*}
                h(f(P)) \leq \log\left(H(P)^dH(f)\binom{N+d}{d}\right) = dh(P)+h(f)+\log{\binom{N+d}{d}}.
            \end{equation*}
            Since $f=[f_0,\ldots,f_N]$ is a morphism, the Nullstellensatz implies there exist $N+1$ sets of polynomials $\{g_{0,j},\ldots, g_{N,j}\} \in R[x_0,\ldots, x_N]$ homogeneous of degree $D - d$ such that
            \begin{equation*}
                \sum_{i=0}^N f_ig_{i,j} = \Res \cdot  x_j^D,
            \end{equation*}
            where $\Res = \text{resultant}(f_0,\ldots,f_N)$ (\cite[\S 1 (p.8)]{Macaulay}).
            Now we compute
            \begin{align*}
            		H(P)^D &= \max(\abs{x_j})^D\\
            			&= \max \abs{\sum_{i=0}^N f_ig_{i,j}}\\
            			&\leq (N+1) \max{\abs{f_ig_{i,j}}}\\
            			&\leq (N+1) \binom{N+D-d}{D-d}(\max H(g_i)) H(P)^{D-d}H(f(P)).
            \end{align*}
            Dividing both sides by $H(P)^{D-d}$ and taking logarithms yields
            \begin{equation*}
            dh(P) \leq h(f(P)) + \log(C_3)
            \end{equation*}
            where
            \begin{equation*}
            		C_3 = \binom{N+D-d}{D-d}(\max_i h(g_i))
            \end{equation*}
            does not depend on $P$. Taking the larger of the upper and lower bounds gives us the desired $C_1$ such that
            \begin{equation}\label{eq_height_difference}
            		\abs{\hhat_f(P) - dh(P)} \leq C_1.
            \end{equation}

            With the $C_1$ in hand, we apply the limit definition of the canonical height
            \begin{equation*}
            		\hhat_f(P) = \lim_{n \to \infty} \frac{h(f^n(P))}{(\deg{f})^n}.
            \end{equation*}
            Let $a \geq 0$ be an integer. Then from (\ref{eq_height_difference}),
            \begin{align*}
            		\abs{\frac{h(f^a(P))}{(\deg{f})^a} - h(P)}
            		&= \abs{ \sum_{k=0}^{a-1} \frac{h(f^{k+1}(P))}{(\deg{f})^{k+1}} - \frac{h(f^k(P))}{(\deg{f})^k} }\\
            		&\leq   \sum_{k=0}^{a-1} \frac{1}{(\deg{f})^{k+1}}\abs{h(f^{k+1}(P)) - (\deg{f})h(f^k(P)) }\\
            		&\leq \sum_{k=0}^{a-1} \frac{C_1}{(\deg{f})^{k+1}}= \frac{1}{\deg{f}}\sum_{k=0}^{a-1} \frac{C_1}{(\deg{f})^k}\\
            \end{align*}
			Taking the limit as $a \to \infty$, we have
            \begin{equation*}
				\abs{\hhat_f(P) - h(P)} \leq
                    \frac{1}{\deg{f}}\sum_{k=0}^{\infty} \frac{C_1}{(\deg{f})^k}
                    =\frac{C_1}{\deg(f)-1}.
            \end{equation*}
        \end{proof}

        Now that we have an upper bound on the height of a rational preperiodic point, we can determine how far we must carry the $p$-adic approximation.  Since the implementation is for $\Q$, we compute the constant needed for the LLL application. We are going to apply LLL to the lattice in $\Z^{N+1}$ generated by the $p$-adic approximation $\overline{P}$ (for $f(P)=P$) and $p^{\ell}$ times the standard basis of $\Z^{N+1}$. Let $b'$ be the smallest vector in the lattice obtained from applying LLL. Let $P$ be the projective point with coordinates $b'$. By our choice of $\ell$, $P$ is unique point of height $\leq B$ corresponding to $\overline{P}$, if such a point exists.

        The constant $B$ is determined in Proposition \ref{prop_height_bound}, so we need to determine the required $\ell$ given $B$.
        \begin{prop}
            Let $\overline{P} \in \Z/p^{\ell}\Z$ and $P$ the point corresponding to the smallest vector from applying the LLL algorithm to the coordinates of $\overline{P}$ and $p^{\ell}$ times the standard basis of $\Z^{N+1}$. If
            \begin{equation*}
                p^{\ell} \geq 2^{N/2+1}B^2\sqrt{N+1},
            \end{equation*}
            then $P$ is the unique point corresponding to $\overline{P}$ of height $< B$ if such a point exists.
        \end{prop}
        \begin{proof}
            It is not hard to see that points in $\P^N(\Q)$ of height $<B$ map injectively into $\P^{N}(\Z/p^{\ell}\Z)$ for $p^{\ell} > 2B^2$.

            The LLL algorithm in Sage uses the parameters $\delta=3/4$ and $\eta=0.501$.
            From \cite[Proposition 1.6]{LLL} we get the bounds on the smallest resulting basis vector
            \begin{equation}\label{eq2}
                \abs{b_0} \leq 2^{N/4}d(L)^{1/(N+1)},
            \end{equation}
            where $d(L)$ is the determinant of the lattice. The determinant of the lattice does not depend on the choice of basis and can be computed as
            \begin{equation*}
                d(L) = \det(b_0,b_1,\ldots,b_N) = \det(\overline{P},p^{\ell},\ldots,p^{\ell}) \geq (p^{\ell})^{N+1}.
            \end{equation*}
            To change from the vector norm to the height, we have
            \begin{align*}
                \abs{b_i} &= \sqrt{\sum{(b_i)_i^2}}\\
                &\leq H(b_i)\sqrt{N+1}.
            \end{align*}
            Returning to (\ref{eq2}), we need
            \begin{equation*}
                2B^2 \sqrt{N+1} \leq 2^{N/4}p^{\ell}
            \end{equation*}
            and thus
            \begin{equation*}
                2B^2 \sqrt{N+1} \leq p^{\ell}.
            \end{equation*}
            The vector resulting from $LLL$ may not actually be the smallest vector in the lattice, so we need correct for that.  From \cite[Proposition 1.11]{LLL}, we have the bounds
            \begin{equation*}
                \abs{b_0}^2 \leq 2^{N}\abs{x}^2
            \end{equation*}
            for any $x$ in the lattice. So we need an additional factor of $2^{N/2}$.

            Therefore, we have
            \begin{align*}
                p^{\ell} \geq 2^{N/2 +1} B^2 \sqrt{N+1}
            \end{align*}
        \end{proof}

        When computing the $p$-adic approximation, if the Jacobian is invertible, we lift with Hensel's Lemma from $p^{\ell}$ to $p^{2\ell}$. If not, we try all possible lifts to $p^{\ell+1}$.

        \begin{lem}[Hensel's Lemma]
            Let $K$ be a number field with ring of integer $\O_K$ and $p$ a prime of $K$. Let $f:\A^{N}(\O_K) \to \A^N(\O_K)$. Suppose there exists $P_0$ such that
            \begin{equation*}
                f(P_0) \equiv (0,\ldots,0) \mod{p^{\ell}}
            \end{equation*}
            and $f'(P_0)$ is invertible. Then there exists a unique $P_1$ such that
            \begin{equation*}
                P = P_0 + P_1p^{\ell},
            \end{equation*}
            as vectors, such that
            \begin{equation*}
                f(P) \equiv (0,\ldots,0) \mod{p^{2\ell}}.
            \end{equation*}
        \end{lem}

    \subsection{Rational preperiodic points}
        At this point, we have computed all the rational periodic points. We must now determine the rational points that are preperiodic but not periodic for $f$.
        \begin{defn}
            We say that $Q$ is an \emph{$n\tth$ preimage} of $P$ by $f$ if
            \begin{equation*}
                f^n(Q) =P.
            \end{equation*}
        \end{defn}
        Since every preperiodic point is eventually periodic, we know that some forward image of every preperiodic point must be periodic. Thus, by computing all the rational preimages of the rational periodic points, we arrive at the full set of rational preperiodic points.

        Given a point $P$, we solve the equation $f(Q)=P$ using elimination theory. Let $I$ be the ideal generated by $f(Q)-P$. Let $G$ be a Groebner basis of $I$ with respect to the lexicographic ordering on $K[X_1,\ldots,X_N]$. Then $G \cap K[X_j,\ldots,X_N]$ is a basis for $I \cap K[X_j,\ldots,X_N]$, the elimination ideals. In particular, we can find $I \cap K[x_N]$ and solve for the possible values of $X_N$. Then we work backwards, one variable at a time, until we arrive at the full set of solutions.
        \begin{exmp}
			For the map $f(z) = z^2 - 7/4$ we already know about the fixed point at infinity $(1:0)$ and the two-cycle $[(1:2) \to (-3:2)]$. We now compute pre-images to find the preperiodic points.
			
			The only rational preimage of $(1:0)$ is $(1:0)$.  The rational preimages of $(1:2)$ are $\{(-3:2),(3:2)\}$. The rational preimages of $(-3:2)$ are $\{(1:2),(-1:2)\}$.
			
			The points $(-1:2)$ and $(3:2)$ were not previously known, so we must find their rational preimages. They both have no rational preimages, so the final set of rational preimages is
			\begin{equation*}
				\{(1:0),(\pm 1:2),(\pm 3:2)\}.
			\end{equation*}
        \end{exmp}

\section{Generalized dynatomic cycles} \label{sect_dyn_cycles}
    In this section we present an object for studying the set of preperiodic points of a given period.
    We first recall the notion of a dynatomic cycle for $f: \P^N \to \P^N$.

    Let $K$ be an algebraically closed field and $f: \P^N_K \to\P^N_K$ be a morphism defined over $K$.  Consider the graph of $f^n$ in the product variety $\P^N \times \P^N$ defined as
    \begin{equation*}
        \Gamma_n = \{(P,f^{n}(P)) \col P \in \P^N\}
    \end{equation*}
    and the diagonal defined as
    \begin{equation*}
        \Delta = \{(P,P) \col P \in \P^N\}.
    \end{equation*}
    Their intersection is precisely the periodic points of period $n$, and we can determine the multiplicity as the multiplicity of the intersection.  Denote the intersection multiplicity of $\Gamma_n$ and $\Delta$ at a point $(P,P) \in \P^N \times \P^N$ to be $a_P(n)$ and, when the intersection is proper, the algebraic zero-cycle of periodic points of period $n$ as
    \begin{equation*}
        \Phi_n(f) = \sum_{P \in \P^N} a_P(n)(P).
    \end{equation*}
    Define
    \begin{equation*}
        a_P^{\ast}(n) = \sum_{d \mid n} \mu\left(\frac{n}{d}\right) a_P(d)
    \end{equation*}
    and
    \begin{equation*}
        \Phi_n^{\ast}(f) = \sum_{d \mid n} \mu\left(\frac{n}{d}\right)
        \Phi_d(f) = \sum_{P \in X} a_{P}^{\ast}(n) (P),
    \end{equation*}
    where $\mu$ is the M\"{o}bius function.
    \begin{defn}
        We call $\Phi^{\ast}_n(\phi)$ the \emph{$n\tth$ dynatomic cycle}. If
        $a_P^{\ast}(n) >0$, then we call $P$ a periodic point of \emph{formal} period $n$.
    \end{defn}

    We now generalize this construction to preperiodic points.
    \begin{defn}
        A preperiodic point $P$ of \emph{period $(m,n)$} satisfies $f^{n+m}(P) = f^m(P)$. We call $m$ the \emph{preperiod} of $P$. Note that a point with period $(m,n)$ is also a point with period $(m+t, kn)$ for any $t,k \in \N$.

        A point $P$ has \emph{minimal period $(m,n)$} if $P$ has preperiod exactly $m$ and $f^m(P)$ has minimal period $n$.

        Define the \emph{generalized $(m,n)$-period cycle} as
        \begin{equation*}
            \Phi_{m,n}(f) = \Gamma_{n+m} \cap \Gamma_m.
        \end{equation*}
        The points in its support have preperiod at most $m$ and their $m\tth$ iterates have period $n$. We are interested in the points with minimal period $(m,n)$. Define
        \begin{align*}
            \Phi_{m,n}^{\ast}(f) &=\sum_{d \mid n} \mu(n/d)\big((\Gamma_{m+d} \cap \Gamma_m) - (\Gamma_{m+d-1} \cap \Gamma_{m-1})\big).
        \end{align*}
    \end{defn}
    \begin{rem}
        For $f:\P^1 \to \P^1$, $\Phi_{m,n}(f)$ and $\Phi^{\ast}_{m,n}(f)$ have the following representations. Let $F_n,G_n$ be homogeneous polynomials such that $f^n = [F_n,G_n]$. Then, $\Phi_{m,n}(f)$ and $\Phi^{\ast}_n(f)$ are polynomials in two variables. We obtain
        \begin{align*}
            \Phi_{m,n}(f) &= \Phi_n(F_m,G_m) = G_mF_{n+m} - F_mG_{n+m}\\
            \Phi^{\ast}_{m,n}(f) &= \frac{\Phi_n^{\ast}(F_m,G_m)}{\Phi_n^{\ast}(F_{m-1},G_{m-1})}.
        \end{align*}
        $\Phi_{m,n}(f)$ is clearly a polynomial, and effectivity (Theorem \ref{thm_preperiodic_dynatomic}\ref{thm_effectitivity_c}) is the statement that $\Phi^{\ast}_{m,n}(f)$ is also a polynomial.
    \end{rem}

    \begin{defn}
        Points whose multiplicity is non-zero in $\Phi^{\ast}_{m,n}(f)$ are called \emph{formal preperiodic points with formal period $(m,n)$}.
    \end{defn}
    \begin{defn}
        For $n >m \geq 1$ we say that $f$ is $(m,n)$-\emph{nondegenerate} if $\Gamma_{k+d}$ and $\Gamma_k$ intersect properly for all $d \mid n$ and $0 \leq k \leq m$, where $\Gamma_0 = \Delta = \{(P,P) \col P \in \P^N\}$, the diagonal.
    \end{defn}

    \begin{thm}\label{thm_preperiodic_dynatomic}
        Let $f:X\to X$ be $(m,n)$-nondegenerate.
        \begin{enumerate}
            \item Points $P$ are in the support of $\Phi_{m,n}(f)$ if and only if $P$ is preperiodic with preperiod at most $m$ and $f^m(P)$ has period $n$.
            \item $\Phi_{m,n}(f) - \Phi_{m-1,n}(f)$ is effective. \label{thm_effectitivity_b}
            \item $\Phi^{\ast}_{m,n}$ is effective. \label{thm_effectitivity_c}
            \item Points $P$ in the support of $\Phi^{\ast}_{m,n}(f)$ satisfy $f^m(P)$ has formal period $n$.
        \end{enumerate}
    \end{thm}
    To prove Theorem \ref{thm_preperiodic_dynatomic} we use methods similar to \cite{Hutz8}. Let $R_P$ be the local ring of the product $(P,P)$ in $\P^N \times \P^N$. We first prove that we need only the naive intersection theory. Serre's definition of intersection theory is the following.
    \begin{equation*}
        i(\Gamma_{n+m},\Gamma_m;P) = \sum_{i=0}^{b -1} (-1)^{i}
        \dim_K(\tor_{i}(R_P/I_{\Gamma_{n+m}},R_P/I_{\Gamma_m})).
    \end{equation*}
    Note that for this definition to work, we must have a representation of $f$ which is defined for all points $\{f^k(P) \col 0 \leq k \leq n+m\}$. In \cite{Hutz8} this was done by replacing $f$ with $f^n$ to be able to work only with fixed points. Since we are dealing with a preperiod, that method is not possible here. However, since we are working with $f:\P^N \to \P^N$ we already have a global representation for the map. Since the set $\{f^k(P) \col 0 \leq k \leq n+m\}$ is a finite set of points, we can find a hyperplane that does not intersect this set and, by conjugating, move this hyperplane to $x_0=0$. Then we can dehomogenize $f$ to $F\left(\frac{x_1}{x_0},\ldots,\frac{x_N}{x_0}\right) = F(X_1,X_2,\ldots,X_N)=F(\textbf{X})$. We can then write the coordinate functions of $F$ as power series in $K[[\textbf{X}]]$.

    \begin{lem}\textup{\cite[Corollary to Theorem V.B.4]{Serre}}\label{lem1}
        Let $(R,\mathfrak{m})$ be a regular local ring of dimension $b$, and let $M$ and $N$ be two
        non-zero finitely generated $R$-modules such that $M \otimes N$ is of finite
        length.  Then $\tor_i(M,N) = 0$ for all $i>0$ if and only if $M$ and $N$ are
        Cohen-Macaulay modules and $\dim M + \dim N = b$.
    \end{lem}

    \begin{prop} \label{prop_Tor_i}
        Let $f:\P^N \to \P^N$ be a morphism defined over $K$ such that $f$
        is $(m,n)$-nondegenerate.  Let $P \in \P^N(K)$.  Then, $\tor_i(R_P/I_{\Gamma_{n+m}},R_P/I_{\Gamma_m}) =0$ for all $i>0$.
    \end{prop}

    \begin{proof}
        We have $\dim{\P^N \times \P^N} = 2N$ and $\dim{\Gamma_{n+m}}=\dim{\Gamma_m}=N$.  The ideals $I_{\Gamma_{n+m}}$ and $I_{\Gamma_m}$ are each generated by $N$ elements and $\Gamma_{n+m}$ and $\Gamma_m$ intersect properly.  Therefore,
        \begin{equation*}
            \dim_K(R_P/(I_{\Gamma_{n+m}} + I_{\Gamma_m})) = \text{length}(R_P/I_{\Gamma_{n+m}} \otimes R_P/I_{\Gamma_m}) < \infty.
        \end{equation*}

        Thus, the union of the generators of $I_{\Gamma_{n+m}}$ and the generators of $I_{\Gamma_m}$ are a system of parameters for $R_P$ \cite[Proposition III.B.6]{Serre}. Consequently, since the local ring $R_P$ is Cohen-Macaulay we can conclude that $R_P/I_{\Gamma_{n+m}}$ is Cohen-Macaulay of dimension $N$ \cite[Corollary to Theorem IV.B.2]{Serre}.  Similarly with $I_{\Gamma_m}$, we conclude that $R_P/I_{\Gamma_n}$ is Cohen-Macaulay of
        dimension $N$.

        We have fulfilled the hypotheses of Lemma \ref{lem1} and can conclude the result.
    \end{proof}
    Proposition \ref{prop_Tor_i} implies that
    \begin{equation*}
        i(\Gamma_{n+m},\Gamma_m;P) = \dim_K(\tor_{0}(R_P/I_{\Gamma_{n+m}},R_P/I_{\Gamma_m})),
    \end{equation*}
    which is the codimension of the ideal $(I_{\Gamma_{n+m}} + I_{\Gamma_m})$.
    Recall that we can compute the codimension of an ideal from its leading term ideal:
    \begin{equation*}
        K[[X_1,\ldots,X_N]]/I \cong_K \Span(X^{v} \mid X^{v} \not\in LT(I)).
    \end{equation*}

    \begin{lem}\textup{\cite[Corollary 6.9]{Eisenbud2}} \label{lem_flatness}
        Suppose that $(R,\mathfrak{m})$ is a local Noetherian ring.  Let $x \in R$ be a non-zero divisor on $R$ and let $M$ be a finitely generated $R$-module.  If $x$ is a non-zero divisor on $M$, then $M$ is flat over $R$ if and only if $M/xM$ is flat over $R/(x)$.
    \end{lem}
    Effectivity is a local property, so we can consider each point $P \in \P^N(K)$ separately. We dehomogenize and move $P$ to the origin. By abuse of notation, we will still consider the power series representation of $f$ at $P$ as $F(\textbf{X})=[F_1(\textbf{X}),\ldots,F_N(\textbf{X})]$. The method is to deform the algebraic zero-cycle to get a zero-cycle where all points are multiplicity $1$.
    We consider $Z_{m,n,P}$ to be the algebraic zero-cycle obtained by intersecting the equations for $\Gamma_{n+m}= (Y_1-F_1^{n+m}(\textbf{X}), \ldots, Y_N-F_N^{n+m}(\textbf{X}))$ and the graph $\Gamma_m = (Y_1-F_1^m(\textbf{X}), \ldots, Y_N-F_N^m(\textbf{X}))$ as analytic varieties (in $R_P$).  We deform $Z_{m,n,P}$ by considering the iterates of
    \begin{equation*}
        F(\textbf{X},t)=[F_1(\textbf{X})+t,\ldots,F_N(\textbf{X})+t]
    \end{equation*}
    for a parameter $t \in \A^1_K$ and their graphs, denoted $\Gamma_n(t)$.  We denote the deformed family as $Z_{m,n,P}(t)$.  Notice that we are deforming and then iterating so that $Z_{m,n,P}(t)$ is associated to $(F(\textbf{X},t))^{n+m}$ and $(F(\textbf{X},t))^m$.
    \begin{prop} \label{prop_flat}
        Let $n,m \in \N$ be such that $f$ is $(m,n)$-nondegenerate and let $P \in \P^N(K)$.  The family $Z_{m,n,P}(t)$ is flat over $K[[t]]$.
    \end{prop}
    \begin{proof}
        Working locally at $P$, $Z_{m,n,P} = a_{P}(m,n)(P)$ with $a_P(m,n) = \dim_K \widehat{R}_P/(I_{\Gamma_{n+m}} + I_{\Gamma_n})$.  Thus, to show flatness for $Z_{m,n,P}(t)$, we need to show flatness for $\widehat{R}_P[[t]]/(I_{\Gamma_{n+m}}(t) + I_{\Gamma_m}(t))$.

        We apply Lemma \ref{lem_flatness} with $M = \widehat{R}_P[[t]]/(I_{\Gamma_{n+m}}(t) + I_{\Gamma_m}(t))$, $R = K[[t]]$, and $x=t$.
        We see that $M/tM \cong \widehat{R}_P/(I_{\Gamma_{n+m}}(t) + I_{\Gamma_m}(t))$ from our choice of deformation and $K[[t]]/(t) \cong K$.  Thus, $M/tM$ is a flat $K$-module since it is a finite dimensional $K$-vector space by the $(m,n)$-nondegeneracy of $f$.  Now, we just need to show that $t$ is not a zero divisor on $M$.

        Assume that $t$ is a zero divisor.  Then, there exists a $b \in M$ with $b \neq 0$ such that $tb =0$.  In particular, there exist $a_i \in K[[t]]$ such that \[tb = \sum_{i=1}^{2N} a_i b_i,\] where $b_i$ are the generators of $(I_{\Gamma_{n+m}}(t) + I_{\Gamma_m}(t))$.  Specializing to $t=0$, we must have \[\left(\sum_{i=1}^{2N} a_i b_i\right)_{t=0} = 0,\] with $(b_i)_{t=0} \neq 0$ for all $i$.  Assume that $(a_i)_{t=0} = 0$ for all $i$; then we have
         \[
            \sum_{i=1}^{2N} \frac{a_i}{t}b_i = b
         \]
         with $\frac{a_i}{t} \in K[[t]]$.  This contradicts $b \not\in (I_{\Gamma_{n+m}(t)} + I_{\Gamma_m}(t))$, so we have at least one $(a_i)_{t=0} \neq 0$. Hence, there is a relation among the $(b_i)_{t=0}$, which contradicts the assumption that $f$ is $(m,n)$-nondegenerate.
    \end{proof}

    \begin{lem} \label{lem2}
        If $d \mid n$, then
        \begin{equation*}
            a_P(m,d) \leq a_P(m,n).
        \end{equation*}
    \end{lem}
    \begin{proof}
        We need to see that $\Gamma_m \cap \Gamma_{m+n} \subseteq \Gamma_m \cap \Gamma_{m+d}$. We obtain $f^{m+n}$ from $f^{m+d}$ by taking
        \begin{equation*}
            f^{m+n} = f^{n-d}\circ f^{m+d}.
        \end{equation*}
        Thus, $\Gamma_{m+n}$ is obtained from $\Gamma_{m+d}$ by taking algebraic combinations of elements of the ideal.
    \end{proof}

    \begin{proof}[Proof of Theorem \ref{thm_preperiodic_dynatomic}]
        \mbox{}
        \begin{enumerate}
            \item Points that are both on the graph $\Gamma_{n+m}$ and $\Gamma_m$ must satisfy $f^{n+m}(P) = f^m(P)$.

            \item We must show that multiplicities of points in $\Phi_{m,n}(f)$ are larger than multiplicities of points in $\Phi_{m-1,n}(f)$. In particular, we need to show that
                \begin{equation*}
                    (I_{\Gamma_{n+m}} + I_{\Gamma_m}) \subseteq (I_{\Gamma_{n+{m-1}}} + I_{\Gamma_{m-1}}).
                \end{equation*}
                To go from the right-hand side to the left-hand side, we replace $X_i$ with $F_i(\textbf{X})$ for $1 \leq i \leq N$.

            \item  We fix $(m,n)$ and consider each point $P \in \P^N(K)$.  If $P$ is not periodic of period $(m,n)$, then we have $a_P(k,d) = 0$ for all $d \mid n$ and $0 \leq k \leq m$ and, hence, $a_P^{\ast}(m,n) = 0$.  So we may assume that $P$ is periodic of period $(m,n)$.

                We consider the family of algebraic zero-cycles $Z_{(m,n),P}(t)$ defined above.  By Proposition \ref{prop_flat} this is a flat family and, thus, by \cite{Lazarsfeld} we have that
                \begin{equation*}
                    \lim_{t \to 0} Z_{(m,n),P}(t) = Z_{(m,n),P}(0) = Z_{(m,n),P}.
                \end{equation*}
                In particular, if $\{P_j(t)\}$ are the points in the support of $Z_{(m,n),P}(t)$ which go to $P$ as $t \to 0$, then if we write the algebraic zero-cycle as
                \begin{equation*}
                    Z_{(m,n),P}(t) = \sum_{P_j(t)} a_{P_j(t)}(m,n) (P_j(t)),
                \end{equation*}
                we have that
                \begin{equation*}
                    a_P(m,n) = \sum_j a_{P_j(t)}(m,n) \quad \text{and} \quad a_P^{\ast}(m,n) = \sum_j a_{P_j(t)}^{\ast}(m,n).
                \end{equation*}
                Note that each $P_j(t)$ is periodic with minimal period $(k_j,d_j)$ with $d_j \mid n$ and $k_j \leq m$. There are finitely many such $P_j(t)$; in fact by flatness, there are $a_P(m,n)$ of them counted with multiplicity.
                From standard results in the theory of analytic varieties in several complex variables concerning the Weierstrass Preparation theorem and multiple roots of Weierstrass polynomials \cite[\S 1.4]{Dangelo}, we know that the set of $t$ values for which there is a solution $P_j(t)$ with multiplicity greater than one is a thin set.
                In particular, generically there are $a_P(m,n)$ distinct $P_j(t)$ which satisfy $P_j(0)=P$.  Finally, $a_P(k,d) \leq a_P(m,n)$ for $d \mid n$ and $0 \leq k \leq m$ by (\ref{thm_effectitivity_b}) and Lemma \ref{lem2}.
                Thus, by avoiding a thin set of $t$, for each $d \mid n$ and $0 \le k \leq m$ each $P_j(t)$ occurs with multiplicity $1$ in $Z_{(k,d),P}(t)$ if its $k\tth$ iterate has minimal period dividing $d$ and preperiod at most $k$ and multiplicity $0$ otherwise.  We compute for $P_j(t)$ with minimal period $(k_j,d_j)$:
                \begin{align*}
                    \Phi_{m,n}^{\ast}(f) &= \sum_{d \mid n} \mu(n/d)\big((\Gamma_d \cap \Gamma_m) - (\Gamma_d \cap \Gamma_{m-1})\big)
                \end{align*}
                \begin{align*}
                    a_{P_j(t)}^{\ast}(m,n) &= \sum_{d \mid n} \mu\left(\frac{n}{d}\right) \left(a_{P_j(t)}(m,d) - a_{P_j(t)}(m-1,d)\right) \\
                    &=\sum_{d \mid \frac{n}{d_j}} \mu\left(\frac{n}{dd_j}\right).
                \end{align*}
                The term $\left(a_{P_j(t)}(m,d) - a_{P_j(t)}(m-1,d)\right)$ is $1$ if $k_j=m$ and $d_j \mid d$ and $0$ otherwise. Then we are taking the M\"obius sum of $0,1$ where it is $1$ if $d_j \mid d$ and $0$ otherwise. Thus,
                \begin{equation*}
                    a_{P_j(t)}^{\ast}(m,n) = 1
                \end{equation*}
                if $P_j(t)$ has minimal preperiod $m$ and $f^m(P_j(t))$ has minimal period $n$. Otherwise, $a_{P_j(t)}^{\ast}(m,n) =0$.

                Since
                \begin{equation*}
                    a_P^{\ast}(m,n) = \sum_j a_{P_j(t)}^{\ast}(m,n),
                \end{equation*}
                then $a_P^{\ast}(m,n) \geq 0$.

            \item As in (\ref{thm_effectitivity_c}), we perturb the system and we know from the proof of (\ref{thm_effectitivity_c}) that $a_{P_j(t)}(m,n) \geq 1$ if and only if $f^m(P_j(t))$ has minimal period $n$. In particular,
                \begin{equation*}
                    a_{f^m(P_j(t))}^{\ast}(n)=1.
                \end{equation*}
                We also have
                \begin{equation*}
                    a_{f^m(P)}^{\ast}(n) = \sum_j a_{f^m(P_j(t))}^{\ast}(n) \geq 1.
                \end{equation*}
                Thus, $f^m(P)$ has formal period $n$.
        \end{enumerate}
    \end{proof}
    For dynatomic cycles it is true that if $P$ is multiplicity $1$ in $\Phi^{\ast}_n(f)$, then $f$ is a point with minimal period $n$ \cite{Hutz1}. This turns out not to be true for generalized dynatomic cycles.
    \begin{exmp}
        Let $f(z) = z^2-1$. Then we compute
        \begin{equation*}
            \Phi^{\ast}_{1,2}(f) = z(z-1).
        \end{equation*}
        However, $0$ is a preperiodic point with minimal period $(0,2)$.
    \end{exmp}

    We now compute the number of preperiodic points of period $(m,n)$ and the number of formal preperiodic points of period $(m,n)$.

    We need to compute the intersection number for $\Gamma_m, \Gamma_{n+m} \subset \mathbb{P}^{N} \times \mathbb{P}^{N}$. Let $D_1$ and $D_2$ be the pullbacks in $\mathbb{P}^{N} \times \mathbb{P}^{N}$ of a hyperplane class $D$ in $\mathbb{P}^{N}$ by the first and second projections, respectively.
    \begin{lem}\cite[Proposition 4.16]{Hutz1}
        For $k \in \N$, the class of $\Gamma_k$ is given by
        \begin{equation*}
            \sum_{j=0}^{N} (d^k)^{N-j}D_1^{N-j}D_2^{j}.
        \end{equation*}
    \end{lem}

    \begin{thm}
        A morphism $f: \mathbb{P}^N \to  \mathbb{P}^N$ of degree $d$ has
        \begin{equation*}
            \deg(\Phi_{m,n}(f)) = \sum_{j=0}^{N} d^{nj+mN}.
        \end{equation*}
        This is the number of preperiodic points of period $(m,n)$ counted with multiplicity.
    \end{thm}

    \begin{proof}
        We compute the intersection number of
        $\Gamma_{n+m}$ and $\Gamma_m$.
        \begin{align*}
            (\Gamma_{n+m}) \cdot(\Gamma_m) &= \left(\sum_{j=0}^{N} (d^{n+m})^{N-j} D_1^{N-j}D_2^{j}\right) \cdot \left(\sum_{k=0}^{N} (d^m)^{N-k}D_1^{N-k}D_2^{k}\right) \\
            &= \sum_{j+k=N}(d^{n+m})^{N-j}(d^m)^{N-k}D_1^{N}D_2^N \\
            &= \sum_{j+k=N} d^{(n+m)(N-j) + m(N-k)}\\
            &= \sum_{j+k=N} d^{n(N-j) + mN}\\
            &= \sum_{j=0}^N d^{nj + mN}.
        \end{align*}
    \end{proof}

    \begin{cor}
        A morphism $f: \mathbb{P}^N \to  \mathbb{P}^N$ given by $N+1$ homogeneous forms of degree $d$ has
        \begin{align*}
            \deg(\Phi^{\ast}_{m,n}(f))
            &= \sum_{D \mid n} \mu(n/D) \left(\deg(\Phi_{m,D}) - \deg(\Phi_{m-1,D})\right)\\
            &= \sum_{D \mid n} \mu(n/D) \sum_{j=0}^N d^{Dj + N}.
        \end{align*}
    \end{cor}

    For $f:\P^1 \to \P^1$, a topic for further study is the geometry of the dynatomic modular curves resulting from the generalized dynatomic polynomials, see \cite{Bousch,Morton4}.

\section{Uniform Boundedness}\label{sect_experimental_results}
  \subsection{The family $z^d+c$}
    Poonen studied the special case of the Morton-Silverman uniform boundedness conjecture for quadratic polynomial maps $f_c(z) = z^2+c$. In this section we apply our algorithm to a computational investigation of the families of maps $f_{d,c} = z^d+c$, which have been studied by Narkiewiscz. The following lemma shows that the denominator of $c$ must be a $d\tth$ power, generalizing an observation made in \cite{Hutz5}.
    \begin{lem}
        Suppose that $z^d+c$ has a periodic point $\alpha\in \A^N(K)$.  Then for each nonarchimedean place $v$ of $K$ with $v(c)<0$, we have $v(c)=dv(\alpha)$.  For each nonarchimedean place with $v(c)\geq 0$, we have $v(\alpha)\geq 0$.
    \end{lem}

    \begin{proof}
         If $0>v(c)>dv(\alpha)$, then, by the ultrametric inequality,
        $v(\alpha^d+c)=dv(\alpha)$ and, in particular, $dv(\alpha^d+c)<v(\alpha^d+c)<v(c)$.  By induction, $v(f_{d,c}^n(\alpha))= d^nv(\alpha)$ which, since $v(\alpha)\neq 0$, contradicts the periodicity of $\alpha$.  If, on the other hand, $0>v(c)$ and $dv(\alpha)>v(c)$, we have $v(\alpha^d+c)=v(c)$.  But in this case, $dv(\phi(\alpha))=dv(c)<v(c)$, and so the previous argument shows that $f_{d,c}(\alpha)$ (and hence $\alpha$) is not periodic.

        For the second claim, simply note that if $v(c)\geq 0$ but $v(\alpha)<0$, we immediately conclude $v(\alpha^d+c)=dv(\alpha)<0$.  By induction we obtain $v(f_{d,c}^n(\alpha))=d^nv(\alpha)$, from which is it clear that $\alpha$ cannot be preperiodic under $f_{d,c}$.
    \end{proof}

    The previous lemma greatly reduces the search space when we apply the algorithm to all maps $f_{d,c}(z)$ defined over $\Q$ with $H(c) < B$ for some height bound $B$. The results are summarized in the following table.

\renewcommand{\arraystretch}{1.4}
    \begin{equation*}
        \begin{tabular}{|l|l|l|l|l|l|}
          \hline
          map & height bound & max period & max \# periodic & max \# preperiodic\\
          \hline
          $z^2+c$ & 1,000,000 & 3, $c= -\frac{29}{16}$ & 5, $c=-\frac{21}{16}$ & 9, $c=-\frac{29}{16}$\\
          \hline
          $z^3+c$ & 1,000,000 & 1, $c=0$ & 4, $c=0$& 4, $c=0$ \\
          \hline
          $z^4+c$ & 5,000,000 & 2, $c=-1$ & 3, $c=-1$& 4, $c=-1$\\
          \hline
          $z^5+c$ & 5,000,000 & 1, $c=0$ & 4, $c=0$& 4, $c=0$ \\
          \hline
          $z^6+c$ & 10,000,000 & 2, $c=-1$ & 3, $c=-1$& 4, $c=-1$\\
          \hline
          $z^7+c$ & 10,000,000 & 1, $c=0$ & 4, $c=0$& 4, $c=0$ \\
          \hline
          $z^8+c$ & 10,000,000 & 2, $c=-1$ & 3, $c=-1$& 4, $c=-1$\\
          \hline
          $z^9+c$ & 10,000,000 & 1, $c=0$ & 4, $c=0$& 4, $c=0$ \\
          \hline
          $z^{10}+c$ & 10,000,000 & 2, $c=-1$ & 3, $c=-1$& 4, $c=-1$\\
          \hline
          $z^{11}+c$ & 10,000,000 & 1, $c=0$ & 4, $c=0$& 4, $c=0$ \\
          \hline
        \end{tabular}
    \end{equation*}

In addition to the $12$ structures from Poonen \cite{Poonen} for $z^2+c$, the following $2$ preperiodic structures are possible. Note that the fixed point at infinity is included in the diagrams.
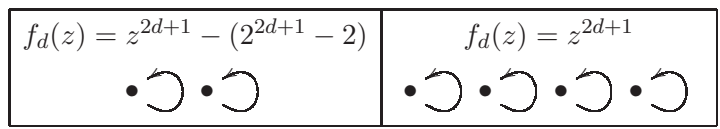
\begin{equation*}
\begin{tabular}{|c|c|}
\hline
$f_d(z) = z^{2d+1}-(2^{2d+1}-2)$
&
$f_d(z)=z^{2d+1}$
\\

$\xygraph{
!{<0cm,0cm>;<1cm,0cm>:<0cm,1cm>::}
!{(0,0) }*+{\bullet}="a"
!{(1,0) }*+{\bullet}="b"
"a":@(rd,ru)"a"
"b":@(rd,ru)"b"
}$

&
$\xygraph{
!{<0cm,0cm>;<1cm,0cm>:<0cm,1cm>::}
!{(0,0) }*+{\bullet}="a"
!{(1,0) }*+{\bullet}="b"
!{(2,0) }*+{\bullet}="c"
!{(3,0) }*+{\bullet}="d"
"a":@(rd,ru)"a"
"b":@(rd,ru)"b"
"c":@(rd,ru)"c"
"d":@(rd,ru)"d"
}$
\\
\hline
\end{tabular}
\end{equation*}

    From this data we make the following conjectures.
    \begin{conj}[Generalized Poonen]
        For $n > 3$ there is no $f_{d,c}(z) = z^d+c$ defined over $\Q$ with a $\Q$-rational periodic point of minimal period $n$. For maps of the form $f_{d,c}$ we have
        \begin{equation*}
            \#\Pre(f_{d,c},\Q) \leq 9.
        \end{equation*}
    \end{conj}
    Note that this conjecture is no longer a special case of the Morton-Silverman conjecture as it allows the degree of map to increase. We provide a more general conjecture along these lines in Section \ref{sect_gen_conj}.

    For odd $d$, results of Narkiewicz resolve Conjecture 1 as a special case. In particular
    \begin{thm}[Narkiewicz \cite{Narkiewicz2}]
        For $n > 1$ and $d$ odd, there is no $c \in \Q$ such that $f_{d,c}$ has a $\Q$-rational periodic point of minimal period $n$. Furthermore,
        \begin{equation*}
            \#\Pre(f_{d,c},\Q) \leq 4.
        \end{equation*}
    \end{thm}
    The proof is actually quite simple. Since $f_{d,c}$ is nondecreasing there can only be fixed points (with no preperiod). The bound comes from counting the number of rational roots of $f_{d,c}(z) -z$.

    The even degree case remain open.
    \begin{conj1a}[Even degree]
        For $n > 2$ there is no even $d >2$ and $c \in \Q$ such that $f_{d,c}$ a $\Q$-rational periodic point of minimal period $n$. Furthermore,
        \begin{equation*}
            \#\Pre(f_{d,c},\Q) \leq 4.
        \end{equation*}
    \end{conj1a}

    \subsection{Families of Conservative Maps}
        A point $P$ is a \emph{critical point} for $f:\P^1 \to \P^1$ if $f'(P) =0$. A map is \emph{conservative} if all its critical points are fixed points. Note that conservative maps are a special case of post-critically finite maps, maps whose critical points are all preperiodic. The algorithm was applied to two families of conservative maps.

        The conservative maps
        \begin{equation*}
            f_d(z) = \frac{(d-2)z^d +dz}{dz^{d-1} + (d-2)}
        \end{equation*}
        for $2 \leq d \leq 100$ were examined. In all cases there were $4$ rational preperiodic points $\{0,1,-1,\infty\}$. For $d$ odd, they are all fixed.  For $d$ even, $\{0,1,\infty\}$ are fixed and $-1$ is strictly preperiodic, $[-1 \to 1]$.

        The conservative maps
        \begin{equation*}
            f_d(z) = \frac{d}{d-1}z + z^d
        \end{equation*}
        for $2 \leq d \leq 200$ were examined. For $d=2$ there are $3$ $\Q$-rational fixed points and $1$ strictly preperiodic point for a total of $4$ rational preperiodic points.  For $3 \leq d \leq 200$, there are 2 fixed points.


        From this data it seems reasonable to conjecture that the number of rational preperiodic points for both of these families of conservative maps is uniformly bounded independent of $d$.

    \subsection{A more general conjecture}\label{sect_gen_conj}
        In addition to families of maps already discussed, the author examined a few other families, such as families of maps of the form $z^d + cz^e$, $d > e \geq 1$ for fixed $e$. In all cases, there seemed to be a similar phenomenon of uniform boundedness independent of $d$. On this somewhat limited evidence, consider the following conjecture.
        \begin{conj}
            Let $g(z)$ be any rational map and let
            \begin{equation*}
                f_{d,g}(z) = z^d + g(z).
            \end{equation*}
            Then there exists a constant $C(g,D)$ such that for all number fields $[K:\Q] =D$
            \begin{equation*}
                \#\Pre(f_{d,g},K) < C.
            \end{equation*}
        \end{conj}
        There are many further questions associated with this conjecture, such as whether $C$ depends only on $\deg{g}$ and not $g$ itself.

\section{Isolated Examples, Running Time, and Next Steps} \label{sect_examples}
    We now give a few interesting isolated examples of the full $\Q$-rational preperiodic structure for maps on $\P^N$. For $\P^1$, most examples are drawn from \cite{Benedetto4, Manes2, Poonen}. For $\P^N$, examples are drawn from \cite{Hutz4} or created from lower dimension examples. The goal was to find a few interesting examples with either long cycles, many connected components, or simply many rational preperiodic points. For $N>1$, it is virtually certain that these examples can be bettered in all three aspects.

    The columns of the chart are
    \begin{itemize}
        \item The coordinates of the morphism $f:\P^N \to \P^N$.
        \item The list of $\Q$-rational cycle lengths.
        \item The number of $\Q$-rational preperiodic points in each connected component, listed in the same order as the cycle lengths.
        \item The total number of $\Q$-rational preperiodic points.
    \end{itemize}

        \renewcommand{\arraystretch}{1.25}
        \begin{equation*}
            \begin{tabular}{|l|l|l|l|}
                \hline
                \multicolumn{4}{|c|}{$\P^1$} \\
                \hline
                map & cycles & \# con. comp. & \# Pre\\
                \hline
                $z^2-1$ & $\{2,1\}$ & $\{3,1\}$ & $4$\\
                \hline
                $z^2 - \frac{7}{4}$ & $\{2,1\}$ & $\{4,1\}$ & 5\\
                \hline
                $\frac{5}{24}z^3 - \frac{53}{24}z+ 1$ &$ \{4,1\}$ & $\{4,1\}$ & 5\\
                \hline
                $z^2 - \frac{3}{4}$ & $\{1,1,1\}$ & $\{2,2,1\}$ & $6$\\
                \hline
                $z^2 -2$ & $\{1,1,1\}$ & $\{3,2,1\}$ & $6$\\
                \hline
                $\frac{1}{12}z^3 - \frac{25}{12}z + 1$  & $\{5,1\}$ & $\{7,1\}$ & $8$\\
                \hline
                $z^2 - \frac{29}{16}$ & $\{3,1\}$ & $\{8,1\}$ & $9$\\
                \hline
                $z^2 - \frac{21}{16}$ & $\{2,1,1,1\}$ & $\{4,2,2,1\}$ & $9$\\
                \hline
                $\frac{4}{30}z^3 - \frac{91}{30}z + 1$ & $\{2,1\}$ & $\{9,1\}$ & $10$\\
                \hline
                $-\frac{5}{4}z + \frac{1}{z}$ & $\{2,1,1,1\}$ & $\{4,2,2,2\}$ & 10\\
                \hline
                $\frac{1}{240}z^3 - \frac{151}{60}z + 1$  & $\{2,2,2,1\}$ & $\{4,4,2,1\}$ & 11\\
                \hline
                $-\frac{3}{2}z^3 + \frac{19}{6}z$  & $\{2,1,1\}$ & $\{10,1,1\}$ & $12$\\
                \hline
                $\frac{7}{24}z - \frac{7}{6z}$  &$ \{4,1\}$ & $\{8,4\}$ & $12$\\
                \hline
            \end{tabular}
        \end{equation*}

        \renewcommand{\arraystretch}{1.25}
        \begin{equation*}
            \begin{tabular}{|l|l|l|l|}
                \hline
                \multicolumn{4}{|c|}{$\P^2$} \\
                \hline
                map & cycles & \# con. comp. & \# Pre\\
                \hline
                $[-\frac{38}{45}x^2 + (2y - \frac{7}{45}z)x + (-\frac{1}{2}y^2 - \frac{1}{2}yz + z^2),$ & $\{9,1\}$ & $\{9,2\}$ & $11$\\
                $-\frac{67}{90}x^2 + (2y + \frac{157}{90}z)x - yz,z^2]$ &&&\\
                \hline
                $[2x^3 - 50xz^2 + 24z^3,$& $\{20,1,1\}$ & $\{28,1,1\}$ & $30$\\
                $5y^3 - 53yz^2 + 24z^3,24z^3]$ &&& \\
                \hline
                $[x^2 - \frac{21}{16}z^2,y^2-2z^2,z^2]$ & $\{2,2,1,1,1,1,1,1,1\}$ & $\{12,8,6,6,4,4,2,1,1\}$ & $44$\\
                \hline
                $[-\frac{3}{2}x^3 + \frac{19}{6}xz^2,$& $\{2,2,2,2,2,2,$ & $\{20,20,20,20,10,10$ & $112$\\
                $\frac{1}{240}y^3 - \frac{151}{60}yz^2 + z^3,z^3]$ & $2,2,2,1,1\}$ & $4,4,2,1,1\}$ & \\
                \hline
            \end{tabular}
        \end{equation*}

        \renewcommand{\arraystretch}{1.25}
        \begin{equation*}
            \begin{tabular}{|l|l|l|l|}
                \hline
                \multicolumn{4}{|c|}{$\P^3$} \\
                \hline
                map & cycles &  \# con. comp. & \# Pre \\
                \hline
                $[-x^3 + \frac{5}{4}xw^2 + w^3,\frac{5}{24}y^3 - \frac{53}{24}yw^2+ w^3$, & $\{60, 1, 1, 1\}$ & $\{84, 1, 1, 1\}$ & $87$\\
                $\frac{1}{12}z^3 - \frac{25}{12}zw^2 + w^3,w^3]$ &&&\\
                \hline
                $[(-y - w)x + (-\frac{13}{30}y^2 + \frac{13}{30}wy + u^2)$, & $\{24,1\}$ & $\{96,1\}$ & $97$\\
                $-\frac{1}{2}x^2 + (-y + \frac{3}{2}w)x + (-\frac{1}{3}y^2 + \frac{4}{3}wy),$ &&&\\
                $-\frac{3}{2}z^2 + \frac{5}{2}zw + w^2,w^2]$&&&\\
                \hline
                $[-\frac{3}{2}x^3 + \frac{19}{6}xz^2$, & $\{2,2,2,2,2,2,2,2,$ & $\{105,105,105,105,75,75,$ & $993$\\
                $\frac{1}{240}y^3 - \frac{151}{60}yz^2 + z^3$, &$2,2,2, 2, 2, 2, 2, 2,$ &  $75,75,45,45,45,45, 21, 21, 15,$ &\\
                $\frac{2}{15}w^3 - \frac{91}{30}wz^2 + z^3,z^3]$ & $2, 2, 1, 1, 1\}$ & $ 15, 9, 9, 1, 1, 1\}$ &\\
                \hline
            \end{tabular}
        \end{equation*}

        \renewcommand{\arraystretch}{1.25}
        \begin{equation*}
            \begin{tabular}{|l|l|l|l|}
                \hline
                \multicolumn{4}{|c|}{$\P^4$} \\
                \hline
                map & cycles & \# con. comp. & \# Pre\\
                \hline
                $[-\frac{38}{45}x^2 + (2y - \frac{7}{45}v)x + (-\frac{1}{2}y^2 - \frac{1}{2}vy + v^2),$ & $\{72,1\}$ & $ \{108,2\}$ & $110$\\
                $-\frac{67}{90}x^2 + (2y + \frac{157}{90}v)x - vy,$&&&\\
                $(-u - v)z + (-\frac{13}{30}u^2 + \frac{13}{30}vu + v^2)$,&&&\\
                $-\frac{1}{2}z^2 + (-u + \frac{3}{2}v)z + (-\frac{1}{3}u^2 + \frac{4}{3}vu),v^2]$&&&\\
                \hline
            \end{tabular}
        \end{equation*}

\subsection{Running time}
    \footnote{All computations were done on OSX 10.8.2 running Sage 5.8 and an Intel i7 3.4Ghz processor and 16 Gb 1333Ghz DDR3 RAM.}
    The degree of the map and the length of the cycles have only a small effect on the running time. The dimension 1 examples above all take less than $1$ second to complete, and 1,000 randomly generated degree $20$ polynomials with coefficients of height at most 1,000 completed in an average of $0.2$ seconds. The dimension 2 examples above completed in $1$-$3$ seconds, and the dimension 3 examples completed in $9$-$20$ seconds. The dimension $4$ example at $680$ seconds is a particularly bad case due to the number and size of the primes needed to reduced the possible period list sufficiently. On the other hand, the simple map
    \begin{align*}
        f&:\P^4 \to \P^4\\
        [x,y,z,u,v] &\mapsto [x^2-2v^2,y^2-v^2,z^2,u^2,v^2]
    \end{align*}
    takes only $10$ seconds to find all 175 rational preperiodic points using the primes $\{2,3,5,7\}$.

    The slowest step in most examples is determining the cycle structure modulo the primes of good reduction. For example, the $\P^4$ example above takes $630$ of the total $680$ seconds to reduce the list of possible periods to $\{1,3,72\}$ using primes $\{13,17,19,23\}$.

    The second limiting factor is the lifting step. If all the eigenvalues of the multiplier matrix are not one, the Hensel lifting from $p^\ell$ to $p^{2\ell}$ is very efficient. However, if even one of the eigenvalues of the multiplier matrix is $1$, then the algorithm tries all possible lifts from $p^{\ell}$ to $p^{\ell +1}$. This happens only for a closed subset of all maps; but for those maps, the running time is significantly worse. For example, the map
    \begin{equation*}
        f(x) = \frac{2x^2 + x}{9x^2 + 2x + 1}
    \end{equation*}
    has the fixed point $0$ with multiplier $1$. Determining the possible periods for primes $p \leq 19$ takes less than $1$ second, but it takes $21$ seconds to determine the preperiodic points.

\subsection{Next steps}
    There are several small improvements that could be made to the implementation, but the main improvement would be to implement the determination of the cycle structure modulo primes in a lower level language such as C.

    While all the results in this article either are valid over number fields or there exist appropriate generalizations for number fields, the author plans to approach the number field problem from a different direction. In a subsequent article, the author will examine the feasibility of reducing the number field case to a problem over $\Q$ through either symmetrization or Weil restriction.


\begin{thebibliography}{BDJ{\etalchar{+}}09}

\bibitem[BDJ{\etalchar{+}}09]{Benedetto4}
Robert Benedetto, Ben Dickman, Sasha Joseph, Ben Krause, Dan Rubin, and Xinwen
  Zhou, \emph{Computing points of small height for cubic polynomials}, Involve
  \textbf{2} (2009), 37--64.

\bibitem[Bou92]{Bousch}
Thierry Bousch, \emph{Sur quelques probl\`emes de dynamique holomorphe}, Ph.D.
  thesis, L'Universit\'e d'Orsay, 1992.

\bibitem[D'A93]{Dangelo}
John~P. D'Angelo, \emph{Several complex variables and the geometry of real
  hypersurfaces}, Studies in Advanced Mathematics, CRC Peess, 1993.

\bibitem[DFK12]{DFK}
John~R. Doyle, Xander Faber, and David Krumm, \emph{Computation of preperiodic
  structures for quadratic polynomials overnumber fields}, arxiv:1111.4963.

\bibitem[Eis04]{Eisenbud2}
David Eisenbud, \emph{Commutative algebra}, Graduate Texts in Mathematics, vol.
  150, Springer-Verlag, 2004.

\bibitem[FPS97]{FPS}
E.V. Flynn, Bjorn Poonen, and Edward Schaefer, \emph{Cycles of quadratic
  polynomials and rational points on a genus 2 curve}, Duke Math. J.
  \textbf{90} (1997), 435--463.

\bibitem[FS10]{Fieker}
Claus Fieker and Damien Stehrl\'e, \emph{Algorithmic number theory}, Lecture
  notes in computer science, vol. 6197, ch.~Short bases of lattices over number
  fields, pp.~151--173, Springer, 2010.

\bibitem[HIng]{Hutz5}
Benjamin Hutz and Patrick Ingram, \emph{Numerical evidence for a conjecture of
  {P}oonen}, Rocky Mountain Journal of Mathematics (forthcoming),
  arXiv:0909.5050.

\bibitem[Hut09a]{Hutz3}
Benjamin Hutz, \emph{A computational investigation of {W}ehler {K3} surfaces},
  New Zealand Journal of Mathematics \textbf{39} (2009), 133--141.

\bibitem[Hut09b]{Hutz2}
\bysame, \emph{Good reduction of periodic points}, Illinois J. Math.
  \textbf{53} (2009), no.~4, 1109--1126.

\bibitem[Hut10a]{Hutz1}
\bysame, \emph{Dynatomic cycles for morphisms of projective varieties}, New
  York J. Math \textbf{16} (2010), 125--159.

\bibitem[Hut10b]{Hutz4}
\bysame, \emph{Rational periodic points for degree two polynomial morphisms on
  projective space}, Acta Arith. \textbf{141} (2010), 275--288.

\bibitem[Hut12]{Hutz8}
\bysame, \emph{Effectivity of dynatomic cycles for morphisms of projective
  varieties using deformation theory}, Proceedings of the AMS \textbf{140}
  (2012), 3507--3514.

\bibitem[Laz77]{Lazard}
Daniel Lazard, \emph{Alg\`ebre lin\'eaire sure $k[x_1,\ldots,x_n]$ et
  \'elimination}, Bulletin de la S.M.F. \textbf{105} (1977), 165--190.

\bibitem[Laz81]{Lazarsfeld}
Robert Lazarsfeld, \emph{Excess intersection of divisors}, Comp. Math.
  \textbf{43} (1981), no.~3, 281--296.

\bibitem[LJL82]{LLL}
A.~Lenstra, H.~Lenstra Jr., and L.~Lov{\'a}sz, \emph{Factoring polynomials with
  rational coefficients}, Math. Ann. \textbf{164} (1982), no.~4, 515--534.

\bibitem[Mac94]{Macaulay}
F.S. Macaulay, \emph{The algebraic theory of modular systems}, Cambridge
  University Press, Cambridge, 1916. Reprinted 1994.

\bibitem[Man08]{Manes2}
Michelle Manes, \emph{$\mathbb{Q}$-rational cycles for degree-2 rational maps
  having an automorphism}, Proc. London Math. Soc. \textbf{96} (2008),
  669--696.

\bibitem[Mer96]{Merel}
Lo{\"{\i}}c Merel, \emph{Bornes pour la torsion des courbes elliptiques sur les
  corps de nombres}, Invent. Math. \textbf{124} (1996), no.~1-3, 437--449.
  \MR{MR1369424 (96i:11057)}

\bibitem[Mor96]{Morton4}
Patrick Morton, \emph{On certain algebraic curves related to polynomial maps},
  Comp. Math. \textbf{103} (1996), 319--350.

\bibitem[Nar12]{Narkiewicz2}
W.~Narkiewicz, \emph{On a class of monic binomials}, to appear (2012).

\bibitem[MS94]{Silverman7}
Patrick Morton and Joseph~H. Silverman, \emph{Rational periodic points of
  rational functions}, Int. Math. Res. Not. \textbf{2} (1994), 97--110.
  \MR{MR1264933 (95b:11066)}

\bibitem[MS95]{Silverman6}
\bysame, \emph{Periodic points, multiplicities, and dynamical units}, J. Reine
  Angew. Math. \textbf{461} (1995), 81--122.

\bibitem[Nar12]{Narkiewiecz2}
W. Narkiewicz, \emph{On a class of monic binomials}, to appear, 2012.

\bibitem[Nor50]{Northcott}
D.G. Northcott, \emph{Periodic points of an algebraic variety}, Ann. of Math.
  \textbf{51} (1950), 167--177.

\bibitem[Poo98]{Poonen}
Bjorn Poonen, \emph{The complete classificiation of rational preperiodic points
  of quadratic polynomials over $\mathbb{Q}$: a refined conjecture}, Math. Z.
  \textbf{228} (1998), no.~1, 11--29.

\bibitem[Ser00]{Serre}
Jean-Pierre Serre, \emph{Local algebra}, Springer-Verlag, 2000.

\bibitem[Sil07]{Silverman10}
Joseph~H. Silverman, \emph{The arithmetic of dynamical systems}, Graduate Texts
  in Mathematics, vol. 241, Springer-Verlag, New York, 2007.

\bibitem[SJ05]{sage}
William Stein and David Joyner, \emph{{SAGE}: System for algebra and geometry
  experimentation}, Communications in Computer Algebra (SIGSAM Bulletin) (July
  2005), {\tt http://www.sagemath.org}.

\bibitem[Sto08]{Stoll3}
Michael Stoll, \emph{Rational 6-cycles under iteration of quadratic
  polynomials}, London Math. Soc. J. Comput. Math. \textbf{11} (2008),
  367--380.

\end{thebibliography}

\end{document}